\documentclass[13pt]{article}  

\usepackage{amssymb}              
\usepackage{amsthm}
\usepackage{amsmath}
\usepackage{eucal}
\usepackage{verbatim}
\usepackage{fancyhdr}
\usepackage{color}
\usepackage{enumerate}
\usepackage{calrsfs}
\usepackage{eufrak}

\newtheorem{theorem}{Theorem}
\newtheorem{lemma}{Lemma}

\makeatletter
\newcommand{\leqnomode}{\tagsleft@true}
\newcommand{\reqnomode}{\tagsleft@false}
\makeatother

\def\({\begin{eqnarray}}
\def\){\end{eqnarray}}
\def\[{\begin{eqnarray*}}
\def\]{\end{eqnarray*}}
\def\part#1#2{\frac{\partial #1}{\partial #2}}

\def\R{\mathbb{R}}
\def\Z{\mathbb{Z}}
\def\N{\mathbb{N}}
\def\d{\mathrm{d}}
\def\tot#1#2{\frac{\d #1}{\d #2}}

\def\I{\mathcal{I}}
\def\K{I}

\def\wtx{\widetilde x}
\def\wx#1#2{{\wtx_{#1}}^{\; #2}}
\def\wxji{\wx{j}{i}}

\def\ta#1#2{\tau_{{#1}{#2}}}
\def\taij{\ta{i}{j}}

\def\c{{\mathfrak{c}}}
\def\s{{\mathfrak{s}}}

\def\upsi{{\underline\psi}}

\def\sprod#1#2{{\left\langle {#1}, {#2} \right\rangle}}

\begin{document}

\title{Optimal condition for asymptotic consensus in the Hegselmann-Krause model with finite speed of information propagation}


\author{Jan Haskovec\footnote{Computer, Electrical and Mathematical Sciences \& Engineering, King Abdullah University of Science and Technology, 23955-6900 Thuwal, KSA. jan.haskovec@kaust.edu.sa}
{ and}
Mauro Rodriguez Cartabia\footnote{IMAS (UBA-CONICET) and Departamento de Matem\'{a}tica Facultad de Ciencias Exactas y Naturales, Universidad de Buenos Aires, Ciudad Universitaria, 1428 Buenos Aires, Argentina. mrodriguezcartabia@dm.uba.ar}}

\maketitle

\begin{abstract}
We prove that asymptotic global consensus is always reached
in the Hegselmann-Krause model with finite speed of information propagation $\c>0$
under minimal (i.e., necessary) assumptions on the influence function.
In particular, we assume that the influence function is globally positive,
which is necessary for reaching global consensus, and such that the agents
move with speeds strictly less than $\c$, which is necessary
for well-posedness of solutions.
From this point of view, our result is optimal.
The proof is based on the fact that the state-dependent delay,
induced by the finite speed of information propagation,
is uniformly bounded.
\end{abstract}


state-dependent delay.


\section{Introduction}\label{sec:Intro}
We study the variant of the Hegselmann-Krause model \cite{HK}
of opinion formation with finite speed of information propagation, introduced in \cite{Has:PAMS}.
The model describes the evolution of $N\in\N$, $N\geq 2$, agents with generalized opinions (positions) represented by vectors $x_i = x_i(t)\in\R^d$, $d\in\N$,
$i\in [N]$, where here and in the sequel we denote $[N] := \{1, \cdots, N\}$.
Information between agents propagates with a fixed finite speed $\c>0$.
Consequently, agent located in $x_i=x_i(t)$ at time $t>0$ observes the
position of the agent $x_j$ at time $t-\taij$, where $\taij$ solves
\(  \label{eq:tau}
   \c \taij(t) = |x_i(t) - x_j(t-\taij(t))|,
\)
i.e., $\taij(t)$ is the time that information needs to travel from location $x_j(t-\taij(t))$
to location $x_i(t)$. 
Dynamics of the vectors $x_i=x_i(t)$ is governed by the system
\(  \label{eq:0}
   \dot x_i = \frac{1}{N-1} \sum_{j=1}^N \psi(|\wxji - x_i|) \left( \wxji - x_i \right), \qquad i \in [N],
\)
where we introduced the notation 
\[ \wxji := x_j(t - \taij(t)). \]
We also introduce the formal notation $\ta{i}{i}:=0$ and $\wx{i}{i}:=x_i(t)$
and, if no danger of confusion, we drop the explicit time dependence,
writing just $x_i$ for $x_i(t)$ and similarly for other quantities.
The strictly positive \emph{influence function} $\psi = \psi(r)$ in \eqref{eq:0}
is assumed uniformly Lipschitz continuous on $[0,\infty)$.
Without loss of generality (by an eventual rescaling of time), we set $\psi(0)=1$,
and we assume $\psi$ to be nonincreasing on $[0,\infty)$;
alternatively, one can replace $\psi$ with its nonincreasing rearrangement
$\Psi(u) := \min_{r\in[0,u]} \psi(r)$ 
in the below proofs, preserving their validity.
Moreover, denoting
\(   \label{cond:s}
    \s := \sup_{r>0} \psi(r)r,
\)
we impose that $\s < \c$.

With the above set of assumptions, well-posedness of solutions of the system \eqref{eq:tau}--\eqref{eq:0}
was established in \cite[Theorem 2.1]{Has:PAMS}, subject to the initial condition
\(  \label{IC:0}
    x_i(t) = x_i^0(t)\qquad\mbox{for } i \in [N],\quad t\leq 0,
\)
where $x_i^0=x_i^0(t)$ are uniformly Lipschitz continuous paths on $(-\infty,0]$,
with Lipschitz constant strictly less than $\c$.
Then, due to \eqref{cond:s}, the (unique) solution trajectories $x_i=x_i(t)$
are uniformly Lipschitz continuous on $[0,\infty)$, with Lipschitz constant $\s$.

In \cite[Section 2]{Has:PAMS}, it was explained that the above set of assumptions,
especially \eqref{cond:s} with $\s<\c$, is necessary for the system \eqref{eq:tau}--\eqref{eq:0}
to be solvable.
Let us point out that, once the initial datum $x_i^0=x_i^0(t)$ is fixed,
the boundedness of $\psi(r)r$ in \eqref{cond:s} is in fact only required
on bounded $r$-intervals of length determined by the radius of the support
of the initial datum. This is a direct consequence of the nonexpansivity of the system \eqref{eq:tau}--\eqref{IC:0},
i.e., the fact that all particle trajectories are uniformly contained within a
compact set defined by the support of the initial datum.
Morever, the nonexpansivity also implies a uniform bound on the delays (see Lemma \ref{lem:tauij} below),
so that the initial datum \eqref{IC:0} only needs to be prescribed
on a compact time interval.

\section{Asymptotic consensus}\label{sec:Cons}
Asymptotic (global) consensus in the context of \eqref{eq:0} is defined as the property
\(  \label{def:cons}
   \lim_{t\to\infty} d_x(t) = 0,
\)
where the group diameter $d_x=d_x(t)$ is given by
   $d_x(t) := \max_{i,j\in[N]} |x_i(t) - x_j(t)|$.
>From the extensive literature on the consensus behavior
of multi-agent systems with delay, we point out \cite{Lu},
where consensus in directed static networks (with both linear and nonlinear coupling)
is proved with arbitrary finite communication delays.
Here, the delays $\tau_{ij}$ are fixed (i.e., independent of time),
which facilitates the construction of a Lyapunov functional.
Unfortunately, this approach fails when $\tau_{ij}$ are functions of time,
as is the case with \eqref{eq:tau}.
An overview of recent results on the consensus behavior of delay \emph{linear}
multi-agent systems can be found in \cite{LiuLiu-book}.
The recent papers \cite{Choi} and, resp., \cite{Paolucci} establish asymptotic consensus in the nonlinear
Hegselmann-Krause model with (a-priori prescribed) variable and, resp., distributed time delay.
However, both papers impose a smallness assumption on the delay,
which is essential for the proof. Consequently, their approach cannot be applied
for our system \eqref{eq:tau}--\eqref{eq:0}.

To our best knowledge, the Hegselmann-Krause model with finite speed of information propagation,
i.e., state-dependent delay,
was only studied in \cite{Has:PAMS} so far.
There, asymptotic consensus was proved for solutions of \eqref{eq:tau}--\eqref{IC:0}
in two situations: in the spatially one-dimensional setting $d=1$, or
in any spatial dimension, but under an additional assumption imposing smallness
of $\s$ compared to the propagation speed $\c$. 
It was hypothesized in \cite{Has:PAMS} that all solutions of \eqref{eq:tau}--\eqref{IC:0}
reach asymptotic consensus, in any spatial dimension $d\geq 1$,
with only the minimal assumptions, i.e., as soon as $\psi$ is globally positive and \eqref{cond:s}
holds with $\s < \c$. The goal of this paper is to demonstrate that this hypothesis indeed holds true.

\begin{theorem}\label{thm:main}
Let the assumptions set in Section \ref{sec:Intro} be verified.
Then all solutions of \eqref{eq:tau}--\eqref{IC:0} reach global asymptotic consensus
in the sense of \eqref{def:cons}.
\end{theorem}

Let us point out that all the assumptions given in Section \ref{sec:Intro} are necessary.
Indeed, the assumption $\s < \c$ in \eqref{cond:s} cannot be relaxed,
since otherwise the well-posedness of solutions of the system \eqref{eq:tau}--\eqref{IC:0}
could not be guaranteed. The same holds for Lipschitz continuity of $\psi$.
Moreover, the global positivity of $\psi$
is necessary for reaching (unconditional) global consensus
(i.e., to prevent formation of several opinion clusters).
>From this point of view, the statement of Theorem \ref{thm:main} is optimal.

\section{Notation and preliminaries}\label{sec:Prel}

We introduce the radius of the agent group, 
$R_x(t) := \max_{i\in[N]} |x_i(t)|$.
Lemma 5.1 of \cite{Has:PAMS} states that $R_x=R_x(t)$ is bounded uniformly in time by the radius of the initial datum,
defined as
\[ 
   R_x^0 := \max_{t\in [-S^0,0]} R_x(t),\qquad \mbox{with } S^0=\frac{d_x(0)}{\c-\s}.
\]
I.e., we have $R_x(t) \leq R_x^0$ for all $t\geq 0$.
Consequently, denoting $\upsi := \frac{1}{N-1} \psi\left(2 R_x^0 \right)$
and $\psi_{ij} := \psi\left( \left| \wxji - x_i \right| \right)$, we have
for all $i, j \in [N]$ and $t\geq 0$,
\(   \label{upsi}
   \frac{\psi_{ij}(t)}{N-1} = \frac{1}{N-1} \psi\left( |x_j(t-\tau_{ij}(t)) - x_i(t)| \right)
      \geq \frac{\psi\left(2 R_x^0 \right)}{N-1} 
      = \upsi,
\)
due to the triangle inequality $|x_j(t-\tau_{ij}(t)) - x_i(t)| \leq |x_j(t-\tau_{ij}(t))| + |x_i(t)| \leq 2 R_x^0$
and the monotonicity of $\psi$.

\begin{lemma}\label{lem:tauij}
Along the solutions of \eqref{eq:tau}--\eqref{eq:0} we have
\(  \label{est:tauij}
   \tau_{ij}(t) \leq \frac{d_x(t)}{\c-\s} \qquad \mbox{for all } i, j \in [N] \mbox{ and } t\geq 0.
\)
\end{lemma}

\begin{proof}
By \eqref{eq:tau} we have
$\c \tau_{ij} =   |\wxji - x_i|$.
On the other hand, due to the $\s$-Lipschitz continuity of $x_j$,
\[  
   |\wxji - x_j|  = |x_j(t-\tau_{ij}) - x_j(t)| \leq \s \tau_{ij}.
\]
Therefore, by the triangle inequality,
\[
   \c \tau_{ij} = |\wxji - x_i| \leq |x_i-x_j| + |\wxji-x_j| \leq d_x(t) + \s\tau_{ij}, 
\]
and \eqref{est:tauij} follows.
\end{proof}

Noting that $d_x(t) \leq 2 R_x(t) \leq 2 R_x^0$ for all $t \geq 0$,
Lemma \ref{lem:tauij} gives
\(  \label{bdd:tauij}
   \tau_{ij}(t) \leq \tau \qquad \mbox{for all } i, j \in [N] \mbox{ and } t\geq 0,
\)
with $\tau:= \frac{2R_x^0}{\c-\s}$.
Then, for $n\in\Z$ we denote the closed time interval $\I_n:= [(n-1)\tau, n\tau]$ and
\(  \label{Kn}
   \K_n := \max_{s,t\in \I_n} \max_{i,j\in[N]} \left| x_i(s) - x_j(t) \right|.
\)
Moreover, for any vector $y\in\R^d$ and $n\in\N$ we denote
\(  \label{def:Mm}
   m_n^y := \min_{s\in\I_n}  \min_{j\in[N]} \sprod{x_j(s)}{y}, \qquad
   M_n^y := \max_{s\in\I_n} \max_{j\in[N]} \sprod{x_j(s)}{y},
\)
where $\sprod{\cdot}{\cdot}$ is the standard dot product in $\R^d$.

\section{Proof of Theorem \ref{thm:main}}\label{sec:Proof}

In light of the uniform bound \eqref{bdd:tauij}, we only need to prove that any solution
$x_i=x_i(t)$, $i\in[N]$, of \eqref{eq:0} converges to global asymptotic consensus in the sense of \eqref{def:cons},
as long as the delays $\tau_{ij} = \tau_{ij}(t)$ are nonnegative, continuous, uniformly bounded functions on $[0,\infty)$.
In particular, we assume that \eqref{bdd:tauij} holds with some $\tau>0$, however, do \emph{not} impose any smallness
assumptions on $\tau$.
The proof of Theorem \ref{thm:main} relies on appropriate generalizations of the techniques developed in \cite{Cartabia}.
We shall work out full details only in significantly deviating parts. 


\begin{lemma}  \label{lem:aux2}
For any $n\in\N$, any fixed vector $y\in\R^d$ and any $i\in[N]$ we have
\(  \label{est:Mm}
   m_n^y \leq \sprod{x_i(t)}{y} \leq  M_n^y \qquad\mbox{for all } t\geq (n-1)\tau,
\)
with $m_n^y$ and $M_n^y$ defined in \eqref{def:Mm}.
Moreover, for $\K_n$ defined in \eqref{Kn}, we have
\(  \label{Kn}
   \K_{n+1} \leq \K_n,
\)
\end{lemma}

\begin{proof}
A slight generalization of the proof of \cite[Lemma 3.3.]{Cartabia}.
\end{proof}

For any vector $y\in\R^d$ we introduce the notation
\(   \label{def:Delta}
   \Delta_{ij}^y(t) := \sprod{x_i(t) - x_j(t)}{y}.
\)

The following Lemma is a generalization of \cite[Lemma 3.4, first part]{Cartabia}.

\begin{lemma} \label{lem:aux5}
For any unit vector $y\in\R^d$ and $i,j\in[N]$ we have
\(   \label{lem:aux5:claim}
   \Delta_{ij}^y((n+2)\tau) \leq e^{-2\tau}\sprod{x_i(t)-x_j(s)}{y} + \left(1 - e^{-2\tau}\right) \K_n
\)
for all $s,t\in [n\tau, (n+2)\tau]$.
\end{lemma}

\begin{proof}
Let us fix the unit vector $y\in\R^d$, then \eqref{def:Mm} gives
\(  
   M_n^y - m_n^y 
      \leq \max_{s,t\in\I_n} \max_{i,j\in[N]} \left|x_i(s)-x_j(t)\right|
      = \K_n. 
      \label{MmKn}
\)
We now fix $t\geq n\tau$ and $i\in[N]$ and claim that for any $s \geq t$ we have
\(  \label{lem:aux5:est_M}
   \sprod{x_i(s)}{y} \leq e^{-(s - t)} \sprod{x_i(t)}{y} + \left(1 - e^{-(s - t)} \right) M_n^y,
\)
and
\(  \label{lem:aux5:est_m}
   \sprod{x_i(s)}{y} \geq e^{-(s - t)} \sprod{x_i(t)}{y} + \left(1 - e^{-(s - t)} \right) m_n^y.
\)
Indeed, since $s\geq n\tau$, we have $s-\tau_{ij}(t) \geq (n-1)\tau$ for any $i,j\in\N$, 
and \eqref{est:Mm} gives
\[
   \sprod{x_j(s-\tau_{ij}(t)) - x_i(s)}{y} \leq M_n^y - \sprod{x_i(s)}{y}.
\]
By one more application of  \eqref{est:Mm}, the right-hand side is nonnegative.
Then, with \eqref{eq:0}, we have
\[
   \tot{}{t} \sprod{x_i(s)}{y} &=& \frac{1}{N-1} \sum_{j\neq i} \psi_{ij}(s) \sprod{x_j(s-\tau_{ij}(s)) - x_i(s)}{y} \\
      &\leq& \frac{1}{N-1} \sum_{j\neq i} \left( M_n^y - \sprod{x_i(s)}{y} \right) 
      = M_n^y - \sprod{x_i(s)}{y},
\]
where we also used the universal bound $\psi_{ij} \leq 1$.
Here and in the sequel, the symbol $\sum_{j\neq i}$ denotes summation over all indices $j\in[N]$ such that $j\neq i$.
Estimate \eqref{lem:aux5:est_M} follows then directly by an application of the Gr\"onwall's inequality.
Estimate \eqref{lem:aux5:est_m} is obtained by the same procedure, replacing the vector $y$ with $-y$
and noting that $M_n^{-y} = - m_n^y$.

Let us now fix $i\in[N]$, $t\in [n\tau, (n+2)\tau]$ and set $s:=(n+2)\tau$ in \eqref{lem:aux5:est_M}, which gives
\[
   \sprod{x_i((n+2)\tau)}{y} &\leq& e^{-((n+2)\tau-t)} \sprod{x_i(t)}{y} + \left( 1 - e^{-((n+2)\tau-t)} \right) M_n^y \\
      &=& e^{-((n+2)\tau-t)} \left( \sprod{x_i(t)}{y} - M_n^y \right) + M_n^y \\
      &\leq& e^{-2\tau} \left( \sprod{x_i(t)}{y} - M_n^y \right) + M_n^y \\
      &=& e^{-2\tau} \sprod{x_i(t)}{y} + \left( 1 - e^{-2\tau} \right) M_n^y,
\]
where we used the fact that $\sprod{x_i(t)}{y} - M_n^y \leq 0$ in the third line.
Similarly, for any $j\in[N]$ and $s\in [n\tau, (n+2)\tau]$, \eqref{lem:aux5:est_m} gives
\[
   \sprod{x_j((n+2)\tau)}{y} \geq e^{-2\tau} \sprod{x_j(s)}{y} + \left( 1 - e^{-2\tau} \right) m_n^y.
\]
Combining the above the inequalities, we have
\[
   \Delta_{ij}^y((n+2)\tau) &=& \sprod{x_i((n+2)\tau) - x_j((n+2)\tau)}{y} \\
     &\leq& e^{-2\tau} \sprod{x_i(t) - x_j(s)}{y} + \left( 1 - e^{-2\tau} \right) \left( M_n^y-m_n^y \right)
\]
for any $s, t\in [n\tau, (n+2)\tau]$.
Using the estimate $M_n^y-m_n^y \leq \K_n$ provided by \eqref{MmKn}, we finally recover \eqref{lem:aux5:claim}.
\end{proof}

\begin{lemma} \label{lem:aux6}
For any $n\in\N$ we have
\( \label{lem:aux6:claim}
   \K_{n+1} \leq e^{-\tau} d_x(n\tau) + \left( 1- e^{-\tau} \right) \K_n.
\)
\end{lemma}

\begin{proof}
A slight modification of the proof of \cite[Lemma 3.4, second part]{Cartabia}.
\end{proof}

\begin{lemma}\label{lem:aux7}
Denote $\alpha := \min\left\{e^{-2\tau},  \left( 1 - e^{-\tau}\right) \upsi \right\}$.
For all $n\geq 2$ we have
\[  
   \K_{n+1} \leq  \left( 1 - e^{-\tau} \alpha \right) \K_{n-2}.
\]
\end{lemma}

\begin{proof}
In the first step we prove that
\(  \label{lem:aux7:eq1}
   d_x(n\tau) \leq \left( 1 - \alpha \right) \K_{n-2}
\)
for all $n\geq 2$.
Obviously, if $d_x(n\tau)=0$, there is nothing to prove.
Otherwise, let us fix $I, J \in[N]$ such that
$d_x(n\tau) = \left| x_I(n\tau) - x_J(n\tau) \right|$
and define
\[
   y := \frac{x_I(n\tau) - x_J(n\tau)}{\left| x_I(n\tau) - x_J(n\tau) \right|}.
\]
Then we have $d_x(n\tau) = \Delta_{IJ}^y(n\tau)$, with $\Delta_{IJ}^y(n\tau)$ defined in \eqref{def:Delta}.
Let us now introduce the following alternative: either there exist $T$, $S\in [(n-2)\tau,n\tau]$ such that
\(  \label{lem:aux7:alt1}
   \sprod{x_I(T)-x_J(S)}{y} < 0,
\)
or we have
\(  \label{lem:aux7:alt2}
   \sprod{x_I(t)-x_J(s)}{y} \geq 0
\)
for all $t$, $s\in [(n-2)\tau,n\tau]$.

If \eqref{lem:aux7:alt1} holds, then claim \eqref{lem:aux5:claim} of Lemma \ref{lem:aux5} gives
\[
   \Delta_{IJ}^y(n\tau) &\leq& e^{-2\tau}\sprod{x_I(T)-x_J(S)}{y} + \left(1 - e^{-2\tau}\right) \K_{n-2}  \\
      &<& \left( 1 - e^{-2\tau} \right) \K_{n-2} \\
      &\leq& \left( 1 - \alpha \right) \K_{n-2},
\]
where the last inequality follows from $\alpha \leq e^{-2\tau}$.
Since $d_x(n\tau) = \Delta_{IJ}^y(n\tau)$, \eqref{lem:aux7:eq1} follows.

On the other hand, let us now assume that \eqref{lem:aux7:alt2} holds for all $t$, $s\in [(n-2)\tau,n\tau]$.
With \eqref{eq:0} we have
\(   \label{lem:aux7:eq2}
   \tot{}{t} \Delta_{IJ}^y(t) &=& \frac{1}{N-1} \sum_{k\neq I} \psi_{Ik}(t) \sprod{x_k(t-\tau_{Ik}(t)) - x_I(t)}{y}  \\
      && \qquad + \frac{1}{N-1} \sum_{k\neq J} \psi_{Jk}(t) \sprod{x_J(t) - x_k(t-\tau_{Jk}(t))}{y}.
      \nonumber
\)
Denoting $(A)$ the first term of the right-hand side above, we calculate
\[
   (A) &=& \frac{1}{N-1} \sum_{k\neq I}\psi_{Ik}(t) \left( \sprod{x_k(t-\tau_{Ik}(t))}{y} - M_{n-1}^y \right) \\
     && \qquad + \frac{1}{N-1} \sum_{k\neq I} \psi_{Ik}(t) \left( M_{n-1}^y - \sprod{x_I(t)}{y} \right) \\
      &\leq& {\upsi}\, \sum_{k\neq I} \left( \sprod{x_k(t-\tau_{Ik}(t))}{y} - M_{n-1}^y \right)  +  M_{n-1}^y - \sprod{x_I(t)}{y},
\]
where we used the fact, provided by Lemma \ref{lem:aux2}, that $\sprod{x_k(t-\tau_{ik}(t))}{y} - M_{n-1}^y \leq 0$ for all $k\in[N]$ and $t\in\I_n$,
combined with \eqref{upsi}, and the universal bound $\psi\leq 1$ combined with $M_{n-1}^y - \sprod{x_i(t)}{y} \geq 0$ for $t\in\I_n$.
Moreover, since $I\neq J$ and, again, due to $\sprod{x_k(t-\tau_{Ik}(t))}{y} - M_{n-1}^y \leq 0$, we have
\[
   \sum_{k\neq I} \left( \sprod{x_k(t-\tau_{Ik}(t))}{y} - M_{n-1}^y \right) \leq  \sprod{x_J(t-\tau_{IJ}(t))}{y} - M_{n-1}^y,
\]
so that we finally arrive at
\[
   (A) \leq  
      \left( 1 - \upsi \right) M_{n-1}^y + \upsi\, \sprod{x_J(t-\tau_{IJ}(t))}{y} - \sprod{x_I(t)}{y}.
\]
Similarly, denoting $(B)$ the second term of the right-hand side of \eqref{lem:aux7:eq2}, we have
\[
   (B) \leq 
     - \left( 1 - \upsi \right) m_{n-1}^y - \upsi\, \sprod{x_I(t-\tau_{JI}(t))}{y} + \sprod{x_J(t)}{y}.
\]
Consequently,
\[
    \tot{}{t} \Delta_{IJ}^y(t) &=& (A)+ (B) \\
       &\leq& \left( 1 - \upsi \right) \left( M_{n-1}^y - m_{n-1}^y \right) \\
       && +\; \upsi\, \bigl[ \sprod{x_J(t-\tau_{IJ}(t))}{y} - \sprod{x_I(t-\tau_{JI}(t))}{y} \bigr]   \\
       && +\; \sprod{x_J(t)}{y} - \sprod{x_I(t)}{y} \\
         &\leq& \left( 1 - \upsi \right) \left( M_{n-1}^y - m_{n-1}^y \right) - \Delta_{IJ}^y(t),
\]
where we used \eqref{lem:aux7:alt2} and the definition \eqref{def:Delta} of $\Delta_{IJ}^y(t)$ in the last line.
Moreover, as in \eqref{MmKn}, we have $M_{n-1}^y-m_{n-1}^y\leq \K_{n-1}$ and, by definition, $\upsi\leq 1$, so that
\[
    \tot{}{t} \Delta_{IJ}^y(t) \leq (1-\upsi) \K_{n-1} - \Delta_{IJ}^y(t).
\]
An application of the Gr\"{o}nwall's inequality on the interval $\I_n=[(n-1)\tau,n\tau]$ gives
\[
   \Delta_{IJ}^y(n\tau) &\leq& e^{-\tau} \Delta_{IJ}^y((n-1)\tau) + (1-\upsi) \left( 1 - e^{-\tau} \right) \K_{n-1}  \\
    &\leq& \left( 1 - \left(1- e^{-\tau} \right) \upsi \right) \K_{n-1},
\]
where we used the inequality $\Delta_{IJ}^y((n-1)\tau) \leq \K_{n-1}$ which holds by Cauchy-Schwarz inequality,
recalling that $y$ is a unit vector.
With the definition of $\alpha$ and the relation $\K_{n-1} \leq \K_{n-2}$ provided by Lemma \ref{lem:aux2}, we arrive at
\[
   \Delta_{IJ}^y(n\tau) \leq \left( 1 - \alpha \right) \K_{n-2},
\]
and since $d_x(n\tau) = \Delta_{IJ}^y(n\tau)$, claim \eqref{lem:aux7:eq1} is proved.

Finally, we combine estimate \eqref{lem:aux6:claim} of Lemma \ref{lem:aux6} and \eqref{lem:aux7:eq1}
to see that
\[
   \K_{n+1} &\leq& e^{-\tau} d_x(n\tau) + \left( 1- e^{-\tau} \right) \K_n \\
      &\leq& e^{-\tau} \left( 1 - \alpha \right) \K_{n-2} + \left( 1- e^{-\tau} \right) \K_n  \\
      &\leq& \left( 1 - e^{-\tau} \alpha \right) \K_{n-2},
\]
where we again used the relation $\K_n \leq \K_{n-2}$, provided by Lemma \ref{lem:aux2}, in the last inequality.
\end{proof}

Finally, 
since we have $ \left(1 - e^{-\tau} \alpha \right) \in (0, 1)$,
Lemma \ref{lem:aux7} together with \eqref{Kn} implies
\[
   \lim_{n\to \infty} \K_n = 0,
\]
exponentially fast. This concludes the proof of Theorem \ref{thm:main}.

\bibliographystyle{amsplain}

\end{document}